\documentclass{article}%
\usepackage{graphicx}
\usepackage{amsmath}
\usepackage{amsfonts}
\usepackage{amssymb}%
\providecommand{\U}[1]{\protect\rule{.1in}{.1in}}
\providecommand{\U}[1]{\protect\rule{.1in}{.1in}}
\providecommand{\U}[1]{\protect\rule{.1in}{.1in}}
\providecommand{\U}[1]{\protect\rule{.1in}{.1in}}
\providecommand{\U}[1]{\protect\rule{.1in}{.1in}}
\providecommand{\U}[1]{\protect\rule{.1in}{.1in}}
\providecommand{\U}[1]{\protect\rule{.1in}{.1in}}
\providecommand{\U}[1]{\protect\rule{.1in}{.1in}}
\providecommand{\U}[1]{\protect\rule{.1in}{.1in}}
\providecommand{\U}[1]{\protect\rule{.1in}{.1in}}
\providecommand{\U}[1]{\protect\rule{.1in}{.1in}}
\providecommand{\U}[1]{\protect\rule{.1in}{.1in}}
\providecommand{\U}[1]{\protect\rule{.1in}{.1in}}
\providecommand{\U}[1]{\protect\rule{.1in}{.1in}}
\providecommand{\U}[1]{\protect\rule{.1in}{.1in}}
\providecommand{\U}[1]{\protect\rule{.1in}{.1in}}
\providecommand{\U}[1]{\protect\rule{.1in}{.1in}}
\providecommand{\U}[1]{\protect\rule{.1in}{.1in}}
\providecommand{\U}[1]{\protect\rule{.1in}{.1in}}
\providecommand{\U}[1]{\protect\rule{.1in}{.1in}}
\providecommand{\U}[1]{\protect\rule{.1in}{.1in}}
\providecommand{\U}[1]{\protect\rule{.1in}{.1in}}
\providecommand{\U}[1]{\protect\rule{.1in}{.1in}}
\providecommand{\U}[1]{\protect\rule{.1in}{.1in}}
\providecommand{\U}[1]{\protect\rule{.1in}{.1in}}
\providecommand{\U}[1]{\protect\rule{.1in}{.1in}}
\providecommand{\U}[1]{\protect\rule{.1in}{.1in}}
\providecommand{\U}[1]{\protect\rule{.1in}{.1in}}
\providecommand{\U}[1]{\protect\rule{.1in}{.1in}}

\newtheorem{theorem}{Theorem}
{}

\newtheorem{conclusion}{Conclusion}

\newtheorem{corollary}{Corollary}

\newtheorem{definition}{Definition}

\newtheorem{lemma}{Lemma}
{}

\newtheorem{remark}{Remark}

\newenvironment{proof}[1][Proof]{\textbf{#1.} }{\ \rule{0.5em}{0.5em}}

\begin{document}

\title{Spectral Expansion for the One-Dimensional Dirac Operator with a
Complex-Valued Periodic Potential}
\author{O. A. Veliev\\{\small \ Dogus University, }\\{\small Esenkent 34755, \ Istanbul, Turkey.}\\\ {\small e-mail: oveliev@dogus.edu.tr}}
\date{}
\maketitle

\begin{abstract}
In this paper, we construct the spectral expansion for the one dimensional
non-self-adjoint Dirac operator $L(Q)$ with a complex-valued periodic
$2\times2$ matrix potential $Q$ in the space $L_{2}^{2}(-\infty,\infty)$. To
this end, we study in detail asymptotic formulas for the Bloch eigenvalues and
Bloch functions that are uniform with respect to the complex quasimomentum, as
well as the essential spectral singularities of $L(Q)$.



\end{abstract}

\section{ Introduction and Preliminary Facts}

In this paper, we construct a spectral expansion for the one-dimensional Dirac
operator $L(Q)$ generated in the space $L_{2}^{2}(-\infty,\infty)$ of
$2$-coordinate complex-valued vector functions $\mathbf{y}(x)\mathbf{=}\left(
\begin{array}
[c]{c}%
y_{1}(x)\\
y_{2}(x)
\end{array}
\right)  $ by the differential expression
\begin{equation}
l(\mathbf{y,}Q)=\left(
\begin{array}
[c]{cc}%
0 & 1\\
-1 & 0
\end{array}
\right)  \mathbf{y}^{^{\prime}}(x)+Q\mathbf{y}(x), \tag{1}%
\end{equation}
where $Q(x)=\left(
\begin{array}
[c]{cc}%
p(x) & q(x)\\
q(x) & -p(x)
\end{array}
\right)  ,$ $p$ and $q$ are $\pi$-periodic, complex-valued functions of
bounded variation on $[0,\pi]$. It is well-known [10, 11, 15] that the
spectrum $\sigma(L(Q))$ of the operator $L(Q)$ is the union of the spectra
$\sigma(L_{t}(Q))$ of the operators $L_{t}(Q)$ for $t\in(-1,1]$ generated in
$L_{2}^{2}[0,\pi]$ by (1) and the boundary condition
\begin{equation}
\mathbf{y}(\pi)=e^{i\pi t}\mathbf{y}(0). \tag{2}%
\end{equation}
However, to construct the spectral expansion for $L(Q),$ we need to consider
the eigenvalues and eigenfunction of $L_{t}(Q)$ for $t\in\mathbb{C}.$ Let
$\mathbf{c}(x,\lambda)=\left(
\begin{array}
[c]{c}%
c_{1}(x,\lambda)\\
c_{2}(x,\lambda)
\end{array}
\right)  $ and $\mathbf{s}(x,\lambda)=\left(
\begin{array}
[c]{c}%
s_{1}(x,\lambda)\\
s_{2}(x,\lambda)
\end{array}
\right)  $ be the solutions of the equation%
\begin{equation}
l(\mathbf{y,}Q)=\lambda\mathbf{y} \tag{3}%
\end{equation}
satisfying the following initial conditions%

\begin{equation}
c_{1}(0,\lambda)=s_{2}(0,\lambda)=1,\text{ }s_{1}(0,\lambda)=c_{2}%
(0,\lambda)=0.\tag{4}%
\end{equation}
Substituting the general solution $x_{1}\mathbf{c}(x,\lambda)+x_{2}$
$\mathbf{s}(x,\lambda)$ into boundary condition (2) and using (4), we obtain
the following system of equations
\[
\left\{
\begin{array}
[c]{c}%
x_{1}(c_{1}(\pi,\lambda)-e^{i\pi t})+x_{2}s_{1}(\pi,\lambda)=0\\
x_{1}c_{2}(\pi,\lambda)+x_{2}(s_{2}(\pi,\lambda)-e^{i\pi t})=0
\end{array}
\right.  .
\]
A number $\lambda$ is an eigenvalue of $L_{t}(Q)$ if and only if it is a root
of the characteristic equation
\begin{equation}
\Delta(\lambda,t)=\left\vert
\begin{array}
[c]{cc}%
c_{1}(\pi,\lambda)-e^{i\pi t} & s_{1}(\pi,\lambda)\\
c_{2}(\pi,\lambda) & s_{2}(\pi,\lambda)-e^{i\pi t}%
\end{array}
\right\vert =0.\tag{5}%
\end{equation}
Using the Wronskian equality $c_{1}s_{2}-c_{2}s_{1}=1,$ we obtain
$\Delta(\lambda,t)=e^{it}(F(\lambda)-2\cos\pi t),$ where $F(\lambda
):=c_{1}+s_{2}.$ For simplicity of notation, $s_{j}(\pi,\lambda)$ and
$c_{j}(\pi,\lambda)$ are denoted by $s_{j}$ and $c_{j},$ respectively. Thus,
the eigenvalues $\lambda(t)$ of $L_{t}(Q)$ are the roots of the equation%
\begin{equation}
F(\lambda)=2\cos\pi t.\tag{6}%
\end{equation}
If
\begin{equation}
F^{^{\prime}}(\lambda(t))\neq0\text{ }\And\text{ }s_{1}(\pi,\lambda
(t))\neq0,\tag{7}%
\end{equation}
then $\lambda(t)$ is a simple eigenvalue, and the corresponding eigenfunction
$\Phi_{t}(x)$ can be written in the form
\begin{equation}
\Phi_{t}(x)=s_{1}(\pi,\lambda(t))\mathbf{c}(x,\lambda(t))+\left(  e^{i\pi
t}-c_{1}(\pi,\lambda(t))\right)  \mathbf{s}(x,\lambda(t)).\tag{8}%
\end{equation}

We now briefly outline the structure of this paper. If $p$ and $q$ are
functions of bounded variation, then it follows from the asymptotic formulas
for $\mathbf{c}(x,\lambda)$ and $\mathbf{s}(x,\lambda)$ (see, for example
[12]) that
\begin{equation}
F(\lambda)=2\cos\pi\lambda+O\left(  \frac{1}{\lambda}\right)  , \tag{9}%
\end{equation}
as $\left\vert \lambda\right\vert \rightarrow\infty.$ Using this result, in
Section 2, we prove the following theorem on asymptotic formulas for eigenvalues:

\begin{theorem}
\textit{\ The eigenvalues of }$L_{t}(Q),$ for $t\in D_{h}(0,1),$%
\textit{\ consist of two sequences }

$\left\{  \lambda_{n,1}(t);\text{ }n\in\mathbb{Z}\right\}  $ and $\left\{
\lambda_{n,1}(t);\text{ }n\in\mathbb{Z}\right\}  $\ which satisfy the
following asymptotic formulas%
\begin{equation}
\lambda_{n,1}(t)=(2n+t)+O(\frac{1}{n}),\text{ }\lambda_{n,2}(t)=(2n-t)+O(\frac
{1}{n}), \tag{10}%
\end{equation}
where $D_{h}(0,1)=D_{h}\backslash\left(  U_{h}(0)\cup U_{h}(1)\right)  ,$
$D_{h}=\{t\in\mathbb{C}:|\operatorname{Im}t|\leq2h,-h\leq\operatorname{Re}%
t\leq2-h\},$ $U_{h}(a)=\{t\in\mathbb{C}:|t-a|<h\}$ and $h$ is a fixed number
from $(0,1/10).$ These formulas are uniform with respect to $t\in D_{h}(0,1)$
\end{theorem}

Theorem 1 together with (8) yields the following asymptotic formulas for eigenfunctions:

\begin{theorem}
\textit{\ If }$t\in D_{h}(0,1),$ then the\textit{ }eigenfunctions
$\Phi_{n,1,t}(x)$ and $\Phi_{n,2,t}(x)$ of $L_{t}(Q),$ defined by (8) and
corresponding to the eigenvalue, $\lambda_{n,1}(t)$ and $\lambda_{n,2}%
(t)$\ \textit{satisfy} the following asymptotic formula%
\[
\Phi_{n,1,t}(x)=\left(
\begin{array}
[c]{c}%
1\\
-i
\end{array}
\right)  e^{i(2n+t)x}+O(\frac{1}{n})\text{ }\And\text{ }\Phi_{n,2,t}%
(x)=\left(
\begin{array}
[c]{c}%
1\\
i
\end{array}
\right)  e^{i(-2n+t)x}+O(\frac{1}{n}).
\]
These formulas are uniform with respect to $t\in D_{h}(0,1)$ and $x\in
\lbrack0,\pi].$
\end{theorem}

Thus, for large values of $n,$ the eigenvalues $\lambda_{n,1}(t),$
$\lambda_{n,2}(t)$ and the eigenfunction $\Phi_{n,1,t},$ $\Phi_{n,2,t}$ of
$L_{t}(Q)$ are asymptotically close, respectively, to the eigenvalues $2\pi
n+t,$ $2\pi n-$ $t$ and eigenfunctions
\[
\left(
\begin{array}
[c]{c}%
1\\
-i
\end{array}
\right)  e^{i(2n+t)x},\text{ }\left(
\begin{array}
[c]{c}%
1\\
i
\end{array}
\right)  e^{i(-2n+t)x}%
\]
of $L_{t}(O),$ where $O$ denotes the $2\times2$ zero matrix. Moreover, it can
be easily verified that the boundary condition (2) is strongly regular for
$t\notin\mathbb{Z}$ in the sense of [13]. For the reader's convenience, we
provide this verification in Section 2 (see Remark 1). Therefore, the root
functions of $L_{t}(Q),$ for $t\in D_{h}(0,1),$ forms a Riesz basis in
$L_{2}^{2}[0,\pi]$ according to [13].

Note that there are many papers on the basis properties of the root functions
of Dirac operators in $L_{2}^{2}[a,b],$ for $-\infty<a<b<\infty,$ under
various boundary conditions (see [1-3, 6-9, 12-14] and the references
therein). Here, we use only the fact that the root functions of the Dirac
operator generated by (1) with strongly regular boundary conditions form a
Riesz basis. This property is used, in Section 4, for the construction of the
spectral expansion of the operator generated by (1) in $L_{2}^{2}%
(-\infty,\infty).$ Therefore, we do not discuss in detail the works devoted to
Dirac operators on a bounded interval.

In Section 4, using the investigation of spectral singularities and essential
spectral singularities (ESS) carried out in Section 3, we construct a spectral
expansion for the operator $L(Q)$ by applying the methods developed in papers
[16--19]. Below, we briefly outline the scheme for constructing the spectral
expansion of the Dirac operator $L(Q).$ It follows from Gelfand's paper [4] on
the spectral expansion of self-adjoint differential operators with periodic
coefficients that, for every $f\in L_{2}^{2}(-\infty,\infty),$ there exists%
\begin{equation}
f_{t}(x)=\sum_{k=-\infty}^{\infty}f(x+k)e^{-ikt}\tag{11}%
\end{equation}
such that
\begin{equation}
f(x)=\frac{1}{2}\int\limits_{(-1,1]}f_{t}(x)dt\text{ }\tag{12}%
\end{equation}
and $f_{t+2}(x)=f_{t}(x).$ Since, as noted above, the $t$-periodic boundary
conditions are strongly regular for all $t\neq0,1$, we can employ the Riesz
basis property of the root functions of $L_{t}(Q)$ for $t\in(-1,1]\backslash
\left\{  0,1\right\}  .$ Moreover, for almost all $t\in(-1,1]$, all
eigenvalues of the operators $L_{t}(Q)$ are simple. Therefore, the system
\begin{equation}
\{\Psi_{n,j,t}:j=1,2;n\in\mathbb{Z}\}\tag{13}%
\end{equation}
of eigenfunctions of $L_{t}(Q)$, normalized in the sense of Remark 1, forms a
Reisz basis in $L_{2}^{2}[0,\pi].$ Consequently, for almost all $t$, the
function $f_{t}$~admits the decomposition
\begin{equation}
f_{t}=\sum_{j=1,2;\text{ }n\in\mathbb{Z}}a_{n,j}(t)\Psi_{n,j,t},\tag{14}%
\end{equation}
where $a_{n,j}(t)=(f_{t},X_{n,j,t})$ and $\left\{  X_{n,,j,t}:j=1,2;n\in
\mathbb{Z}\right\}  $ is the biorthogonal system corresponding to (13). That
is,
\[
X_{n,j,t}=\frac{1}{\alpha_{n,j}(t)}\Psi_{n,j,t}^{\ast},\text{ }\alpha
_{n,j}(t)=\left(  \Psi_{n,j,t}^{\ast},\Psi_{n,,jt}\right)  ,
\]
where $\Psi_{n,,j,t}^{\ast}$ denotes the normalized eigenfunction of the
adjoint operator $(L_{t}(Q))^{\ast}$ .

Using (14) in (12), we obtain
\begin{equation}
f=\frac{1}{2}\int\limits_{(-1,1]}\sum_{j=1,2;\text{ }n\in\mathbb{Z}}%
a_{n,j}(t)\Psi_{n,j,t}dt. \tag{15}%
\end{equation}
To derive the spectral expansion in terms of the parameter $t$ from (15), it
is necessary to justify the term-by-term integration. This has been a
challenging and delicate problem since the 1950s. Moreover, in general,
term-by-term integration in (15) is not possible, since the expression
\begin{equation}
a_{n,j}(t)\Psi_{n,j,t}=\frac{1}{\alpha_{n,j}(t)}(f_{t},\Psi_{n,j,t}^{\ast
})\Psi_{n,j,t} \tag{16}%
\end{equation}
may fail to be integrable over $(-1,1]$ for certain values of $j$ and $n.$

Therefore, the first step is to investigate the integrability of (16). Since
this integrability depends on that of $1/\alpha_{n,j}(t)$, we introduce the
following notions, defined independently of the choice of $f,$ which will be
used in the construction of the spectral expansion.

\begin{definition}
We say that a point $\lambda_{0}\in\sigma(L_{t_{0}}(Q))\subset\sigma(L(Q))$ is
an essential spectral singularity (ESS) of the operator $L(Q)$ if there exist
$j=1,2$ and $n\in\mathbb{Z}$ such that $\lambda_{0}=:\lambda_{n,j}(t_{0})$ and
for each $\varepsilon$ the function $\frac{1}{\alpha_{n,j}}$ is not integrable
on $(t_{0}-\varepsilon,t_{0}+\varepsilon)$.
\end{definition}

Thus, in Section 3 we consider the ESS. Then, using these results, in Section
4 we construct the spectral expansion for the Dirac operator $L(Q).$
Throughout the paper, $m_{0},m_{1},...$ denote positive constants that do not
depend on $t,$ $\lambda$ and $x.$ They are used in the sense that, for some
inequality, there exists a constant $m_{i}$ such that the inequality holds.

\section{Uniform Asymptotic Formulas for $t\in D_{h}(0,\pi)$}

It is well-known (see [21]) that if $f$ is a function of bounded variation on
$[0,\pi],$ then
\begin{equation}
\int_{0}^{x}f(t)\cos\lambda tdt=O\left(  \frac{1}{\lambda}\right)  ,\text{
}\int_{0}^{x}f(t)\sin\lambda tdt=O\left(  \frac{1}{\lambda}\right)  , \tag{17}%
\end{equation}
as $\lambda\rightarrow\infty$. Moreover, $O\left(  \frac{1}{\lambda}\right)  $
is independent of $x\in\lbrack0,\pi];$ that is, the estimates in (17) hold
uniformly with respect to $x\in\lbrack0,\pi].$ Therefore, by applying Theorem
1 and Theorem 2 of [12], we obtain the following asymptotic formulas, which
are uniform with respect to $x\in\lbrack0,\pi]$:
\begin{equation}
s_{1}(x,\lambda)=-\sin\lambda x+O\left(  \frac{1}{\lambda}\right)  ,\text{
}s_{2}(x,\lambda)=\cos\lambda x+O\left(  \frac{1}{\lambda}\right)  \tag{18}%
\end{equation}
and
\begin{equation}
c_{1}(x,\lambda)=\cos\lambda x+O\left(  \frac{1}{\lambda}\right)  ,\text{
}c_{2}(x,\lambda)=\sin\lambda x+O\left(  \frac{1}{\lambda}\right)  . \tag{19}%
\end{equation}
Thus, (9) follows from (18) and (19). Now, using (9) and Rouch\'{e}'s theorem,
we prove that there exist positive constants $M(h)$ and $N(h)$ such that
equations (6) and%
\begin{equation}
\cos\pi\lambda-\cos\pi t=0 \tag{20}%
\end{equation}
\ have the same number of zeros inside the circle
\[
C(2n\pm t,\frac{M(h)}{n})=\left\{  \mu\in\mathbb{C}:\left\vert \mu-(2n\pm
t)\right\vert =\frac{M(h)}{n}\right\}
\]
for all $\left\vert n\right\vert >N(h)$ and $t\in D_{h}(0,1).$ For this we
prove that
\begin{equation}
\left\vert F(\lambda)-2\cos\pi\lambda\right\vert <\left\vert 2\cos\pi
\lambda-2\cos\pi t\right\vert \tag{21}%
\end{equation}
for all $\lambda\in C(2n\pm t,\frac{M(h)}{n}),$ that is, for
\[
\lambda=2n\pm t+\frac{M(h)}{n}e^{i\alpha},\text{ }\alpha\in\lbrack0,2\pi].
\]
By (9), there exists $m_{0}>0,$ such that the left-hand side of (21) satisfies
the estimate
\begin{equation}
\left\vert F(\lambda)-2\cos\pi\lambda\right\vert <\frac{m_{0}}{n}, \tag{22}%
\end{equation}
for $\lambda\in C(2\pi n\pm t,\frac{M(h)}{n}),$ for all $\left\vert
n\right\vert >N(h)$ and for all $t\in D_{h}(0,1).$

We now estimate the right-hand side of (21). Using the Taylor expansion of
$\cos\pi\lambda$ about $2n\pm t$, we obtain
\[
\left\vert 2\cos\pi\lambda-2\cos\pi t\pm\frac{\pi M(h)}{n}e^{i\alpha}\sin\pi
t\right\vert <\frac{m_{1}}{n^{2}}%
\]
for some positive constant $m_{1}.$ On the other hand, if $t\in D_{h}(0,1)$
and $h\in(0,1/10)$, then, using the Maclaurin series for $\sin\pi t,$ we
obtain $\left\vert \sin t\right\vert >h/2.$ Therefore, we have
\[
\left\vert 2\cos\pi\lambda-2\cos\pi t\right\vert >\frac{m_{0}}{n},
\]
for all $\lambda\in C(2\pi n\pm t,\frac{M(h)}{n}),$ $\left\vert n\right\vert
>N(h)$, $t\in D_{h}(0,1)$ and for some $M(h).$ Combining this with (22), we
obtain (21). Thus, by Rouch\'{e}'s theorem, there exists a sufficiently large
number $N(h)$ such that, for all $\left\vert n\right\vert >$ $N(h)$, equations
(20) and (6) have the same number of zeros inside the circles $C(2\pi
n+t,\frac{M(h)}{n})$ and $C(2\pi n-t,\frac{M(h)}{n}).$ Since, for $t\in
D_{h}(0,1),$ equation (20) has a unique simple zero inside each of these
circles, equation (6) also has a unique simple zero inside each of them.

In the same way, we prove that for $t\in D_{h}(0,1),$ both operators
$L_{t}(O)$ and $L_{t}(Q)$ have the same number of eigenvalues, namely
$2(2N(h)+1)$, lying outside all circles $C(2\pi n\pm t,\frac{M(h)}{n})$ for
$\left\vert n\right\vert >N(h).$ The eigenvalues of $L_{t}(O)$ outside these
circles are $2\pi n\pm t$ for $\left\vert n\right\vert \leq N(h).$ Thus, the
set of all eigenvalues of $L_{t}(O)$ is
\[
\left\{  \lambda_{n,j}^{0}(t):j=1,2;n\in\mathbb{Z}\right\}  ,
\]
where $\lambda_{n,1}^{0}(t)=2n+t$ and $\lambda_{n,2}^{0}(t)=2n-t$. For
convenience of notation, we denote the set of all eigenvalues of $L_{t}(Q)$
by
\[
\left\{  \lambda_{n,j}(t):j=1,2;n\in\mathbb{Z}\right\}  ,
\]
where $\lambda_{n,j}(t)$ satisfies asymptotic formulas (10). This completes
the proof of Theorem 1.

Let $\nu_{1},\nu_{2},...$ be the roots of the entire function $\frac
{dF}{d\lambda},$ and let $t=t_{k}$ $(k=1,2,...)$ be the points in the set
$D_{h}$ for which $\nu_{s}\in\sigma(L_{t_{k}}(q))$ for some $s.$ It follows
from Theorem 1 that the sequence $\left\{  t_{k}\right\}  _{k=1}^{\infty}$ has
no accumulation point in $D_{h}(0,\pi).$ Therefore, $\left\{  t_{k}\right\}
_{k=1}^{\infty}$ can accumulate only at $0$ and $1.$\ Moreover, it follows
from the asymptotic formula
\begin{equation}
s_{1}(\pi,\lambda_{n,1}(t))=-\sin\pi t+O\left(  \frac{1}{n}\right)  ,\text{
}s_{1}(\pi,\lambda_{n,2}(t))=\sin\pi t+O\left(  \frac{1}{n}\right)  , \tag{23}%
\end{equation}
(see Theorem 1 and (18)) that there exists a constant $m_{2}$ such that
\begin{equation}
\left\vert s_{1}(\pi,\lambda_{n,j}(t))\right\vert >m_{2}h \tag{24}%
\end{equation}
for\textit{ \ }$t\in D_{h}(0,1),$\textit{\ }$j=1,2$ and $\mid n\mid>N(h).$ On
the other hand, by using (19) instead of (18) and repeating the proof of (23),
we obtain
\begin{equation}
c_{1}(\pi,\lambda_{n,j}(t))=\cos\pi t+O\left(  \frac{1}{n}\right)  \tag{25}%
\end{equation}
for $j=1,2.$ Therefore, we have
\[
\frac{e^{i\pi t}-c_{1}(\pi,\lambda_{n,1}(t))}{s_{1}(\pi,\lambda_{n,1}%
(t))}=-i+O\left(  \frac{1}{n}\right)  ,\text{ }\frac{e^{i\pi t}-c_{1}%
(\pi,\lambda_{n,2}(t))}{s_{1}(\pi,\lambda_{n,2}(t))}=i+O\left(  \frac{1}%
{n}\right)  .
\]
This formula, together with (8), (18), and (19), completes the proof of
Theorem 2.

\begin{remark}
The operator $\left(  L_{t}(Q)\right)  ^{\ast},$ adjoint to $L_{t}(Q),$ is
$L_{\overline{t}}(\overline{Q}),$ where
\[
\overline{Q(x)}=\left(
\begin{array}
[c]{cc}%
\overline{p(x)} & \overline{q(x)}\\
\overline{q(x)} & -\overline{p(x)}%
\end{array}
\right)  .
\]
The systems
\[
\left\{  \frac{1}{\sqrt{2\pi}}\left(
\begin{array}
[c]{c}%
1\\
-i
\end{array}
\right)  e^{i(2n+t)x},\text{ }\frac{1}{\sqrt{2\pi}}\left(
\begin{array}
[c]{c}%
1\\
i
\end{array}
\right)  e^{i(-2n+t)x}:n\in\mathbb{Z}\right\}  \text{ }%
\]
and
\[
\left\{  \frac{1}{\sqrt{2\pi}}\left(
\begin{array}
[c]{c}%
1\\
-i
\end{array}
\right)  e^{i(2n+\overline{t})x},\text{ }\frac{1}{\sqrt{2\pi}}\left(
\begin{array}
[c]{c}%
1\\
i
\end{array}
\right)  e^{i(-2n+\overline{t})x}:n\in\mathbb{Z}\right\}  ,
\]
which consist of eigenfunctions of the operators $L_{t}(O)$ and $L_{\overline
{t}}(O),$ respectively, are biortogonal in $L_{2}^{2}[0,\pi].$ For real $t,$
these systems coincide and form an orthonormal basis. However, for nonreal
$t,$ they have norm one in the space $L_{2}^{2}[0,\pi]$ after multiplication
by $e^{-itx}$ and $e^{-i\overline{t}x},$ respectively. Therefore, we denote by
$\Psi_{n,j,t}$ the eigenfunction of $L_{t}(Q)$ satisfying the following
normalization conditions
\[
\parallel e^{-itx}\Psi_{n,j,t}\parallel=1,\text{ }\arg\left(  \Psi
_{n,j,t},\left(
\begin{array}
[c]{c}%
e^{i(2n+t)x}\\
i(-1)^{j}e^{i(2n+t)x}%
\end{array}
\right)  \right)  =0
\]

\end{remark}

\begin{remark}
Boundary condition (2) can be written in the form
\[
A\mathbf{y}(0)+B\mathbf{y}(\pi)=0,
\]
where $A=\left(
\begin{array}
[c]{cc}%
e^{i\pi t} & 0\\
0 & e^{i\pi t}%
\end{array}
\right)  ,$ $B=\left(
\begin{array}
[c]{cc}%
-1 & 0\\
0 & -1
\end{array}
\right)  .$ Let $J_{ij}$ be the determinant built of the $i$-th and $j$-th
columns of the matrix $\left(
\begin{array}
[c]{cccc}%
e^{i\pi t} & 0 & -1 & 0\\
0 & e^{i\pi t} & 0 & -1
\end{array}
\right)  .$ Then, we have
\[
J_{14}=-e^{i\pi t},J_{32}=J_{23}-=-e^{i\pi t},J_{13}=J_{24}=0,J_{12}=e^{2i\pi
t},J_{3,4}=1.
\]
Therefore, $J_{14}+J_{32}\pm i(J_{42}-J_{13})=-2e^{i\pi t}\neq0$ which means
that boundary condition (2) is regular for all $t\in\mathbb{C}.$ Moreover, for
the discriminant of the equation
\[
\left[  J_{14}+J_{32}-i(J_{13}+J_{2,4})\right]  z^{2}+\left[  J_{12}%
+J_{34}\right]  z+\left[  J_{14}-J_{23}\pm i(J_{13}+J_{24})\right]  =0,
\]
we have $(e^{2i\pi t}-1)^{2}\neq0$ for $t\neq k,$ where $k\in\mathbb{Z}.$
Thus, boundary condition (2) is strongly regular for $t\neq k$ (see [13]).
Hence, decomposition (14) holds.
\end{remark}

\begin{remark}
For each $n$ and $j$, $\lambda_{n,j}(t)$ is a simple eigenvalue for all $t\in
D_{h}$ except for a finite number of point $t_{1},t_{2},...,t_{m}$. By
applying the implicit function theorem to (6), we can choose the indices $n$
and $j$ so that, if $\lambda_{n,j}(t_{0})$ is a simple eigenvalue, then
$\lambda_{n,j}(t)$ remains a simple eigenvalue in some neighborhood of $t$ and
depends analytically on $t$ in this neighborhood. It then follows that the
functions $\frac{1}{\alpha_{n,j}(t)}$ and $a_{n,j}(t)\Psi_{n,j,t}$ are defined
on $D_{h}$ except at the finite set of points $t_{1},t_{2},...,t_{m}.$
Moreover, these functions are pairwise continuous on a curve $\gamma\subset
D_{h}$ and continuous on any connected subset of $\gamma\backslash\left\{
t_{1},t_{2},...,t_{m}\right\}  .$ Since the integrals of these functions over
the curves $\gamma$ and $\gamma\backslash\left\{  t_{1},t_{2},...,t_{m}%
\right\}  $ coincide, by the integral over $\gamma$ we mean the integral over
$\gamma\backslash\left\{  t_{1},t_{2},...,t_{m}\right\}  .$
\end{remark}

\section{Spectral singularities and ESS}

Recall that the spectral singularity of $L(Q)$ is a point of its spectrum
$\sigma(L(Q))$ in neighborhood on which the projections of $L(Q)$ are not
uniformly bounded or equivalently a point $\lambda\in\sigma(L(Q))$ is called a
spectral singularity of $L(Q)$ if the spectral projections of the operators
$L_{t}(Q)$ for $t\in(-1,1]$ corresponding to the eigenvalues lying in the
small neighborhood of $\lambda$ are not uniformly bounded. More precisely,
here we use the following definition of the spectral singularity (see [17]).

\begin{definition}
We say that $\lambda\in\sigma(L(Q))$ is a spectral singularity of $L(Q)$ if
for all $\varepsilon>0,$\ there exists a sequence $\left\{  \gamma
_{n}\right\}  $\ of closed contours $\gamma_{n}\subset\{z\in\mathbb{C}:\mid
z-\lambda\mid<\varepsilon\}$ such that for each $n$ the contour $\gamma_{n}$
encloses at most one eigenvalue (counting multiplicity) of $L_{t}(Q)$ for all
$t\in(-1,1]$ and
\[
\lim_{n\rightarrow\infty}\sup_{t}\parallel e(t,\gamma_{n})\parallel=\infty,
\]
where $e(t,\gamma_{n})$ is defined by
\[
e(t,\gamma_{n})=\int_{\gamma_{n}}\left(  L_{t}(Q)-\lambda I\right)
^{-1}d\lambda\text{ }%
\]
and $\sup$ is taken over all $t$ for which $\gamma_{n}$ lies in the resolvent
set of $L_{t}(Q)$.
\end{definition}

It is well-known that if $\lambda_{n,j}(t)$ is a simple eigenvalue of
$L_{t}(Q)$ and $e(\lambda_{n,j}(t))$ is the spectral projection defined by the
contour integration of the resolvent of $L_{t}(Q)$ over the closed contour
$\gamma$ enclosing only the eigenvalue $\lambda_{n,j}(t),$ then%
\begin{equation}
e(\lambda_{n,j}(t))f=\text{ }\frac{1}{\alpha_{n,j}(t)}(f,\Psi_{n,j,t}^{\ast
})\Psi_{n,j,t},\text{ }\left\Vert e(\lambda_{n}(t))\right\Vert =\left\vert
\frac{1}{\alpha_{n,j}(t)}\right\vert . \tag{26}%
\end{equation}
Thus, we need to consider
\begin{equation}
\alpha_{n,j}(t)=\left(  \Psi_{n,j,t},\Psi_{n,j,t}^{\ast}\right)  \tag{27}%
\end{equation}
It is clear that, if (7) for $\lambda(t)=\lambda_{n,j}(t)$ holds, then
\begin{equation}
\Psi_{n,j,t}=\frac{\Phi_{n,j,t}(x)}{\left\Vert \Phi_{n,j,t}(x)\right\Vert
},\text{ }\Psi_{n,j,t}^{\ast}=\frac{\overline{\Phi_{n,j,-t}(x)}}{\left\Vert
\Phi_{n,j,-t}(x)\right\Vert }, \tag{28}%
\end{equation}
where
\[
\Phi_{n,j,\pm t}(x)=s_{1}\mathbf{c}(x,\lambda_{n,j}(\pm t))+\left(  e^{\pm
i\pi t}-c_{1}\right)  \mathbf{s}(x,\lambda_{n,j}(\pm t)).
\]
Using (27), (28) and the following equalities
\begin{align*}
s_{2}+c_{1}  &  =2\cos t,\text{ }c_{1}s_{2}-c_{2}s_{1}=1,\\
\left(  e^{i\pi t}-c_{1}\right)  +\left(  e^{-i\pi t}-c_{1}\right)   &
=s_{2}-c_{1},\left(  e^{i\pi t}-c_{1}\right)  \left(  e^{-i\pi t}%
-c_{1}\right)  =-s_{1}c_{2},
\end{align*}
we obtain $\alpha_{n,j}(t)=$
\begin{equation}
\frac{\int\limits_{0}^{\pi}s_{1}^{2}\left(  c_{1}^{2}(x)+c_{2}^{2}(x)\right)
+s_{1}\left(  s_{2}-c_{1}\right)  (c_{1}(x)s_{1}(x)+c_{2}(x)s_{2}%
(x))-s_{1}c_{2}\left(  s_{1}^{2}(x)+s_{2}^{2}(x)\right)  dx}{\left\Vert
s_{1}\mathbf{c}(x,\lambda_{n,j}(t))+\left(  e^{i\pi t}-c_{1}\right)
\mathbf{s}(x,\lambda_{n,j}(t))\right\Vert \left\Vert s_{1}\mathbf{c}%
(x,\lambda_{n,j}(t))+\left(  e^{-i\pi t}-c_{1}\right)  \mathbf{s}%
(x,\lambda_{n,j}(t))\right\Vert }. \tag{29}%
\end{equation}

On the other hand, using the method of variation of constant, from the
equations
\[
\left(
\begin{array}
[c]{cc}%
0 & 1\\
-1 & 0
\end{array}
\right)  \frac{\partial^{2}\mathbf{c}(x,\lambda)}{\partial x\partial\lambda
}+\left(
\begin{array}
[c]{cc}%
p(x) & q(x)\\
q(x) & -p(x)
\end{array}
-\lambda I\right)  \frac{\partial\mathbf{c}(x,\lambda)}{\partial\lambda
}=\mathbf{c}(x,\lambda),
\]
and
\[
\left(
\begin{array}
[c]{cc}%
0 & 1\\
-1 & 0
\end{array}
\right)  \frac{\partial^{2}\mathbf{c}(x,\lambda)}{\partial x\partial\lambda
}+\left(
\begin{array}
[c]{cc}%
p(x) & q(x)\\
q(x) & -p(x)
\end{array}
-\lambda I\right)  \frac{\partial\mathbf{c}(x,\lambda)}{\partial\lambda
}=\mathbf{c}(x,\lambda),
\]
which are obtained from (3) by differentiation with respect to $\lambda,$we
conclude that%
\[
\frac{\partial\mathbf{c}(x,\lambda)}{\partial\lambda}=s(x,\lambda
)\int\limits_{0}^{x}\left(  c_{1}^{2}(t)+c_{2}^{2}(t)\right)  dt-\mathbf{c}%
(x,\lambda)\int\limits_{0}^{x}s_{1}(c_{1}(t)s_{1}(t)+c_{2}(t)s_{2}(t))dt
\]
and
\[
\frac{\partial\mathbf{s}(x,\lambda)}{\partial\lambda}=s(x,\lambda
)\int\limits_{0}^{x}(c_{1}(t)s_{1}(t)+c_{2}(t)s_{2}(t))dt-c(x,\lambda
)\int\limits_{0}^{x}\left(  s_{1}^{2}(t)+s_{2}^{2}(t)\right)  dt.
\]
Instead of $x$ writing $\pi$, we obtain
\[
F^{^{\prime}}(\lambda)=\int\limits_{0}^{\pi}s_{1}\left(  c_{1}^{2}%
(x)+c_{2}^{2}(x)\right)  +\left(  s_{2}-c_{1}\right)  (c_{1}(x)s_{1}%
(x)+c_{2}(x)s_{2}(x))-c_{2}\left(  s_{1}^{2}(x)+s_{2}^{2}(x)\right)  dx.
\]
These calculations were carried out in [5] for the self-adjoint case and for
$q(x)=0.$ Since we cannot directly refer to [5] and for the sake of
independent reading, we present these calculations here. Thus, the last
formula together with (29) implies that $\alpha_{n,j}(t)=$
\begin{equation}
\frac{-s_{1}F^{^{\prime}}(\lambda_{n,j}(t))}{\left\Vert s_{1}\mathbf{c}%
(x,\lambda_{n,j}(t))+\left(  e^{i\pi t}-c_{1}\right)  \mathbf{s}%
(x,\lambda_{n,j}(t))\right\Vert \left\Vert s_{1}\mathbf{c}(x,\lambda
_{n,j}(t))+\left(  e^{-i\pi t}-c_{1}\right)  \mathbf{s}(x,\lambda
_{n,j}(t))\right\Vert }. \tag{30}%
\end{equation}
Using this, we prove the following theorem, which is similar to Theorem 2.4.3
of [20].

\begin{theorem}
If $t_{0}\neq0,1$ and $\lambda_{0}$ is a multiple eigenvalue of $L_{t_{0}%
}(Q),$ then $\lambda_{0}$ is a spectral singularity of $L(Q)$ and is not an ESS.
\end{theorem}

\begin{proof}
Let $\lambda_{0}$ be a multiple eigenvalue of $L_{t_{0}}(Q)$ of multiplicity
$m,$ where $t_{0}\neq0,1.$ To prove this theorem, we establish the following
relations:
\begin{equation}
F^{^{\prime}}(\lambda)\sim(\lambda_{0}-\lambda)^{m-1}\text{ }\And(\lambda
_{0}-\lambda_{n,j}(t))=(t-t_{0})^{\frac{1}{m}}, \tag{31}%
\end{equation}
as $\lambda\rightarrow\lambda_{0}$ and $t\rightarrow t_{0},$ where
$\lambda_{0}=\lambda_{n,j}(t_{0})$. Here $f(x)\sim g(x)$ as $x\rightarrow
x_{0}$ means that \ $f(x)=O(g(x))$ and $g(x)=O(f(x))$ as $x\rightarrow x_{0}.$
Using the Taylor expansion of $F^{\prime}(\lambda)$ at $\lambda_{0}$ and
taking into account that
\begin{equation}
F^{(k)}(\lambda_{0})=0,\text{ }F^{^{(m)}}(\lambda_{0})\neq0 \tag{32}%
\end{equation}
for $k=1,2,....(m-1),$ we obtain the first relation in (31).

Now, using the Taylor expansion of $F^{\prime}(\lambda)$ at $\lambda_{0}$ and
$\cos t$ at $t_{0},$ together with the equality $F(\lambda_{0})=2\cos t_{0}$
and (32), we obtain
\[
F(\lambda)=2\cos t_{0}+F^{^{(m)}}(\lambda_{0})(\lambda-\lambda_{0}%
)^{m}(1+o(1)),
\]
as $\lambda\rightarrow\lambda_{0}$ and
\[
2\cos t=2\cos t_{0}-\left(  \sin t_{0}\right)  (t-t_{0})-(t-t_{0})^{2}\left(
\tfrac{1}{2}+o(1)\right)  ,
\]
as $t\rightarrow t_{0}.$ These equalities, together with the identity
$F(\lambda_{k,j}(t))=2\cos t$ yield the second relation in (31). Combining the
first and second relations in (31), we obtain the following.
\begin{equation}
F(\lambda_{k,j}(t))-F(\lambda_{k,j}(t_{0}))\sim(t-t_{0})^{\frac{m-1}{m}}.
\tag{33}%
\end{equation}
Therefore, using (30), and arguing as in the proof of Theorem 2.4.3 of [20],
we obtain%
\begin{equation}
\frac{1}{\alpha_{n,j}(t)}\sim(t-t_{0})^{\frac{m-1}{m}} \tag{34}%
\end{equation}
as $t\rightarrow t_{0},$ provided that $t_{0}\neq0,1$ and $\lambda_{k,j}%
(t_{0})$ is a multiple eigenvalue of $L_{t_{0}}(Q)$ of multiplicity $m.$ The
proof of the theorem now follows from (26), Definitions 1 and 2.
\end{proof}

It is clear that if $\lambda_{k,j}(t_{0})$ is a simple eigenvalue, then there
exists $\varepsilon>0$ such that $\lambda_{k,j}(t)$ remains a simple
eigenvalue, and hence $\alpha_{n,j}(t)\neq0$ for all $t\in\left(
t_{0}-\varepsilon,t_{0}+\varepsilon\right)  .$ Therefore, we have the
following consequence of Theorem 3 and Definitions 1 and 2:

\begin{corollary}
The set of spectral singularities is a subset of the set of multiple Bloch
eigenvalues. The set of ESS is a subset of the set of spectral singularities.
Any ESS of $L(Q)$ is a multiple eigenvalue of either $L_{0}(Q)$ or $L_{1}(Q).$
\end{corollary}

Note that most results of [20] concerning spectral singularities and ESS for
the Hill operator can be proved in a similar way for the Dirac operator
$L(Q).$ Here, we record only the results that are essential for the
construction of the spectral expansion for $L(Q).$Corollary 1 will also be
used for this purpose in the next section.

\section{Spectral Expansions}

For simplicity of the technical estimates, we assume that $f$ is a compactly
supported continuous function. Then $f_{t}(x)$ is an analytic function of $t$
in a neighborhood of $D_{h}$ for each $x$. Hence, by the Cauchy's theorem and
(12), for all curve $l$ lying in $D_{h}$ and connecting the points $-1+h$ and
$1+h,$ we have
\begin{equation}
f(x)=\frac{1}{2}\int_{l}f_{t}(x)dt. \tag{35}%
\end{equation}
We chose the curve $l$ so that (14) holds for all $t\in l$, and hence
\begin{equation}
f=\frac{1}{2}\int\limits_{l}\left(  \sum_{j=1,2;\text{ }n\in\mathbb{Z}}%
a_{n,j}(t)\Psi_{n,j,t}\right)  dt. \tag{36}%
\end{equation}
To construct the spectral expansion, we need to consider the integrability of
$a_{n,j}(t)\Psi_{n,j,t}$ and term-by term integrability in (36) for some curve
$l$. This curve is constructed as follows. Let $A$ be the set of $t\in D_{h}$
for which
\begin{equation}
F^{^{\prime}}(\lambda_{n,j}(t))\neq0\text{ }\And\text{ }s_{1}(\pi
,\lambda_{n,j}(t))\neq0. \tag{37}%
\end{equation}
Since $F^{^{\prime}}(\lambda)$ and $s_{1}(\pi,\lambda)$ are entire function
and the asymptotic formulas (9) and (18) hold, the set $A$ consist of the
sequence $t_{1},t_{2},...$ whose accumulation points are $0$ and $1.$ Thus,
\begin{equation}
A=\left\{  t_{k}:k\in\mathbb{N}\right\}  \text{ }\And\overline{A}%
=A\cup\left\{  0,1\right\}  . \tag{38}%
\end{equation}

By Corollary 1, only the eigenvalues $\lambda_{n,j}(0)$ and $\lambda_{n,j}%
(\pi)$ may became the ESS, which prevents integrability. Therefore, we choose
the curve of integration so that it passes only around the points $0$ and
$\pi$. Namely, we construct the curve of integration as follows. Let $h$ be
positive number such that $h\notin\overline{A}$ and
\begin{equation}
\left(  \gamma(0,h)\cup\gamma(1,h)\right)  \cap A=\varnothing\text{,} \tag{39}%
\end{equation}
where $\gamma(0,h)$ and $\gamma(\pi,h)$ are the semicircles
\begin{equation}
\gamma(0,h)=\{\left\vert t\right\vert =h,\operatorname{Im}t\geq0\},\text{
}\gamma(1,h)=\{\left\vert t-1\right\vert =h,\operatorname{Im}t\geq0\}.
\tag{40}%
\end{equation}
Define $l(h)$ by
\begin{equation}
l(h)=B(h)\cup\gamma(0,h)\cup\gamma(1,h)), \tag{41}%
\end{equation}
where $B(h)=[h,1-h]\cup\lbrack1+h,2-h].$ Thus, $l(h)$ consist of the intervals
$[h,1-h]$ and $[1+h,2-h]$ , together with the semicircles defined in (40).

Since the accumulation points of $A$ are $0$ and $1,$ the set $B(h)$ may
contain only a finite number of points from $A.$ Denote these points by
$t_{1},t_{2},...,t_{s}$ and set
\begin{equation}
E(h)=B(h)\backslash\left\{  t_{1},t_{2},...,t_{s}\right\}  . \tag{42}%
\end{equation}
Note that, by integrals over $B(h),$ we mean integrals over $E(h)$ (see Remark
3). Instead of $l$ in (36) using $l(h)=B(h)\cup\gamma(0,h)\cup\gamma(\pi,h)$,
we obtain%
\begin{equation}
f=\frac{1}{2}\left(  \int\limits_{E(h)}f_{t}(x)dt+\int\limits_{\gamma
(0,h)}f_{t}(x)dt+\int\limits_{\gamma(\pi,h)}f_{t}(x)dt\right)  , \tag{43}%
\end{equation}
where
\begin{align}
\int\limits_{E(h)}f_{t}(x)dt  &  =\int\limits_{E(h)}\sum\limits_{j=1,2;k\in
\mathbb{Z}}a_{k,j}(t)\Psi_{k,j,t}(x)dt,\nonumber\\
\int\limits_{\gamma(0,h)}f_{t}(x)dt  &  =\int\limits_{\gamma(0,h)}%
\sum\limits_{j=1,2;k\in\mathbb{Z}}a_{k,j}(t)\Psi_{k,j,t}(x)dtdt,\tag{44}\\
\int\limits_{\gamma(1,h)}f_{t}(x)dt  &  =\int\limits_{\gamma(1,h)}%
\sum\limits_{j=1,2;k\in\mathbb{Z}}a_{k,j}(t)\Psi_{k,j,t}(x)dtdt.\nonumber
\end{align}
Therefore, the construction of the spectral expansion consists of the
following steps:

\textbf{Step 1. }We prove that, for each $x\in\lbrack0,\pi],$ the expression
$a_{k,j}(t)\Psi_{k,j,t}(x)$ is integrable with respect to $t$ over $E(h),$
$\gamma(0,h)$ and $\gamma(1,h)$ (Theorem 4).

\textbf{Step 2. }We prove that, the series in (44) can be integrated term by
term (Theorem 5).

\textbf{Step 3. }We replace, the integral over $\gamma(0,h)$ and $\gamma
(1,h)$, by the integrals over $[-h,h]$ and $[1-h,1+h],$ respectively (Theorem 6).

To prove Theorem 4, we first establish the following lemma:

\begin{lemma}
For each $j=1,2$ and $k\in\mathbb{Z}$ there exists $M$ such that
\begin{equation}
\left\vert \Psi_{n,j,t}(x)\right\vert \leq M, \tag{45}%
\end{equation}
for all $x\in\lbrack0,\pi]$ and $t\in\left(  D_{h}\backslash\overline
{A}\right)  ,$ where
\begin{equation}
\left\vert \Psi_{n,j,t}(x)\right\vert ^{2}=\frac{\left\vert s_{1}%
\mathbf{c}(x,\lambda_{n,j}(t))+\left(  e^{i\pi t}-c_{1}\right)  \mathbf{s}%
(x,\lambda_{n,j}(t))\right\vert ^{2}}{\left\Vert s_{1}\mathbf{c}%
(x,\lambda_{n,j}(t))+\left(  e^{i\pi t}-c_{1}\right)  \mathbf{s}%
(x,\lambda_{n,j}(t))\right\Vert ^{2}}. \tag{46}%
\end{equation}

\end{lemma}

\begin{proof}
We estimate the numerator $A(x,\lambda)$ and denominator $B(\lambda)$ of the
fraction in (46) separately. Using the inequality $\left\vert a\right\vert
^{2}+\left\vert b\right\vert ^{2}\geq2\left\vert ab\right\vert $, we obtain
\[
\left\vert A(x,\lambda)\right\vert \leq2(\left\vert s_{1}\mathbf{c}%
(x,\lambda)\right\vert ^{2}+\left\vert \left(  e^{i\pi t}-c_{1}\right)
\mathbf{s}(x,\lambda)\right\vert ^{2}).
\]
On the other hand, for any compact subset $K$ of the complex plane, there
exists a positive number $M_{1}$ such that%
\[
\left\vert \mathbf{c}(x,\lambda)\right\vert ^{2}\leq M_{1},\text{ }\left\vert
\mathbf{s}(x,\lambda)\right\vert ^{2}\leq M_{1}%
\]
for all $x\in\lbrack0,\pi]$ and $\lambda\in K.$ Therefore, we have
\begin{equation}
\left\vert s_{1}\mathbf{c}(x,\lambda_{n,j}(t))+\left(  e^{i\pi t}%
-c_{1}\right)  \mathbf{s}(x,\lambda_{n,j}(t))\right\vert ^{2}\leq
2M_{1}(\left\vert s_{1}(\lambda_{n,j}(t))\right\vert ^{2}+\left\vert \left(
e^{i\pi t}-c_{1}(\lambda_{n,j}(t))\right)  \right\vert ^{2}) \tag{47}%
\end{equation}
for all $x\in\lbrack0,\pi]$ and $t\in\left(  D_{h}\backslash\overline
{A}\right)  .$

To estimate the denominator $B(\lambda)$, we use the following obvious
inequality%
\[
\frac{\left\vert (\mathbf{c}(\cdot,\lambda),\mathbf{s}(\cdot,\lambda
))\right\vert }{\left\Vert \mathbf{c}(\cdot,\lambda)\right\Vert \left\Vert
\mathbf{s}(\cdot,\lambda)\right\Vert }=:m(\lambda)<1,
\]
due to lineally independence of the solutions $\mathbf{c}(x,\lambda)$ and
$\mathbf{s}(x,\lambda).$ It is clear that $m(\lambda)$ continuously depend on
$\lambda.$ Therefore, for any compact subset $K$ of the complex plane, there
exists a positive number $m_{3}=m(\lambda_{0})<1,$ such that
\[
\frac{\left\vert (\mathbf{c}(\cdot,\lambda),\mathbf{s}(\cdot,\lambda
))\right\vert }{\left\Vert \mathbf{c}(\cdot,\lambda)\right\Vert \left\Vert
\mathbf{s}(\cdot,\lambda)\right\Vert }\leq m_{3}%
\]
for all $\lambda\in K.$ Then, we have
\[
B(\lambda)\geq\left\vert s_{1}\right\vert ^{2}\left\Vert \mathbf{c}%
(\cdot,\lambda)\right\Vert ^{2}+\left\vert e^{i\pi t}-c_{1}\right\vert
^{2}\left\Vert \mathbf{s}(\cdot,\lambda)\right\Vert -2\left\vert s_{1}\left(
e^{i\pi t}-c_{1}\right)  \right\vert m_{3}\left\Vert \mathbf{c}(x,\lambda
)\right\Vert \left\Vert \mathbf{s}(x,\lambda)\right\Vert ^{2}\geq
\]%
\[
(1-m_{3})\left\vert s_{1}\right\vert ^{2}\left\Vert \mathbf{c}(\cdot
,\lambda)\right\Vert ^{2}+\left\vert e^{i\pi t}-c_{1}\right\vert
^{2}\left\Vert \mathbf{s}(\cdot,\lambda)\right\Vert ^{2}.
\]
On the other hand, there exists $m_{4}$ such that
\[
\left\Vert \mathbf{c}(\cdot,\lambda)\right\Vert ^{2}\geq m_{4}\text{ }%
\And\left\Vert \mathbf{s}(\cdot,\lambda)\right\Vert ^{2}\geq m_{4}%
\]
for all $\lambda\in K.$ Thus, we have
\[
\left\Vert s_{1}\mathbf{c}(x,\lambda_{n,j}(t))+\left(  e^{i\pi t}%
-c_{1}\right)  \mathbf{s}(x,\lambda_{n,j}(t))\right\Vert ^{2}\geq
(1-m_{3})m_{4}(\left\vert s_{1}(\lambda_{n,j}(t))\right\vert ^{2}+\left\vert
\left(  e^{i\pi t}-c_{1}(\lambda_{n,j}(t))\right)  \right\vert ^{2})
\]
for all $t\in\left(  D_{h}\backslash\overline{A}\right)  .$ This inequality,
with (47) and (46) implies (45).
\end{proof}

Now we are ready to consider \textbf{Step 1.} More precisely, we prove the following.

\begin{theorem}
Let $f$ be a continuous, compactly supported function. Then, for each
$x\in\lbrack0,\pi],$ the expression $a_{k,j}(t)\Psi_{k,j,t}(x)$ is integrable
with respect to $t$ over the sets $E(h),\gamma(0,h)$ and $\gamma(1,h).$
Moreover, if $\lambda_{n,j}(0)$ and $\lambda_{n,j}(1)$ are not ESS, then this
expression is integrable over $(-h,h)$ and $(1-h,1+h)$, respectively.
\end{theorem}

\begin{proof}
If $f$ \ is a continuous, compactly supported function, then, using the
definitions of $a_{n,t}(t)$ and $f_{t}$ (see (14) and (11)), together with the
Schwarz inequality, we obtain that there exists a positive number $m_{5}$ such
that
\begin{equation}
\left\vert a_{n}(t)\Psi_{n,j,t}\right\vert \leq m_{5}\left\vert \tfrac
{1}{\alpha_{n}(t)}\right\vert ,\forall t\in\left(  D_{h}(0,1)\backslash
\overline{A}\right)  . \tag{48}%
\end{equation}
On the other hand, if $\lambda_{n,j}(t_{0})$ is a simple eigenvalue, then
$\alpha_{n,j}(t_{0})\neq0$ and $\frac{1}{\alpha_{n,j}(t)}(f_{t},\Psi
_{n,j,t}^{\ast})\Psi_{n,j,t}$ is a continuous function in some neighborhood
of$\ t_{0}$. This implies that $a_{n,j}(t)\Psi_{n,j,t}$ depends continuously
on $t$ along the curves $\gamma(0,h)$ and $\gamma(1,h).$ Moreover, this
expression depends continuously on each subinterval of $E(h).$ Note that, by
(42), $E(h)$ consists of a finite number of intervals. Thus, for each fixed
$x$, $a_{n,j}(t)\Psi_{n,j,t}(x)$ is piecewise continuous with respect to $t$
on the set $E(h)\cup\gamma(0,h)\cup\gamma(1,h).$ Since (48) holds, $\left(
E(h)\cup\gamma(0,h)\cup\gamma(1,h)\right)  \subset\left(  D_{h}\backslash
\overline{A}\right)  $ and $\left\vert \tfrac{1}{\alpha_{n}(t)}\right\vert $
is integrable on $E(h)\cup\gamma(0,h)\cup\gamma(1,h)$, it follows that
$a_{k,j}(t)\Psi_{k,j,t}(x)$ is integrable with respect to $t$ over the sets
$E(h),$ $\gamma(0,h)$ and $\gamma(1,h).$

If $\lambda_{n,j}(0)$ and $\lambda_{n,j}(1)$ are not ESS, then by Definition
1, $\tfrac{1}{\alpha_{n}}$ is integrable on $(-h,h)$ and $(1-h,1+h),$
respectively. Therefore, by repeating the above argument for the integrability
of $a_{n,j}(t)\Psi_{n,j,t}(x)$ over $E(h)$, we complete the proof of the theorem.
\end{proof}

From (34), (48) and Remark 3, we obtain the following consequence of this Theorem.

\begin{corollary}
The functions $\frac{1}{\alpha_{n,j}}$ and $a_{n,j}(t)\Psi_{n,j,t}$ are
integrable on $[\varepsilon,1-\varepsilon]\cup\lbrack1+\varepsilon
,2-\varepsilon]$ for any $\varepsilon>0.$
\end{corollary}

To prove that the series in (44) can be integrated term by term (see
\textbf{Step2}) we use the following asymptotic formulas for the
eigenfunctions $\Psi_{n,j,t}(x)$ and $\Psi_{n,j,t}^{\ast}(x)$ of $L_{t}(Q)$
and $\left(  L_{t}(Q)\right)  ^{\ast}=L_{\overline{t}}(\overline{Q})$ and the
following lemma. If $t\in E(h)\cup\gamma(0,h)\cup\gamma(1,h),$ then it follows
from Theorem 2 and Remark 1 that
\begin{equation}
\Psi_{n,j,t}(x)=\frac{1}{\sqrt{2\pi}\left\Vert e^{itx}\right\Vert }\left(
\begin{array}
[c]{c}%
1\\
(-1)^{j}i
\end{array}
\right)  e^{i((-1)^{j-1}2n+t)x}+h_{n,j}(x) \tag{49}%
\end{equation}
and
\begin{equation}
\Psi_{n,j,t}^{\ast}(x)=\frac{1}{\sqrt{2\pi}\left\Vert e^{itx}\right\Vert
}\left(
\begin{array}
[c]{c}%
1\\
(-1)^{j}i
\end{array}
\right)  e^{i((-1)^{j-1}2n+\overline{t})x}+h_{n,j,t}^{\ast}(x), \tag{50}%
\end{equation}
where
\begin{equation}
\left\Vert h_{n,j,t}\right\Vert =O(\frac{1}{n}),\text{ }\left\Vert
h_{n,j,t}^{\ast}\right\Vert =O(\frac{1}{n}) \tag{51}%
\end{equation}
for $j=1,2$ as $\left\vert n\right\vert \rightarrow\infty$ and $O(\frac{1}%
{n})$ does not depend on $t\in E(h)\cup\gamma(0,h)\cup\gamma(1,h).$ Using
these formulas, we estimate the remainder
\[
R_{N}(x,t)=\sum_{j=1,2;\text{ }\mid k\mid>N}a_{k,j}(t)\Psi_{k,j,t}(x)
\]
of the series in (44). Namely, we prove the following.

\begin{lemma}
There exist a positive constants $N(h)$ and $m_{6},$ independent of $t,$ such
that
\begin{equation}
\parallel R_{n}(\cdot,t)\parallel^{2}\leq m_{6}\sum_{j=1,2;\text{ }\mid
k\mid>n}\left\vert \left(  f_{t},\frac{1}{\sqrt{2\pi}\left\Vert e^{itx}%
\right\Vert }\left(
\begin{array}
[c]{c}%
1\\
(-1)^{j}i
\end{array}
\right)  e^{i((-1)^{j-1}2n+\overline{t})x}\right)  \right\vert ^{2}+\frac
{1}{n} \tag{52}%
\end{equation}
for $n>N(h)$ and $t\in E(h)\cup\gamma(0,h)\cup\gamma(1,h).$
\end{lemma}

\begin{proof}
To prove (52) first we prove the inequality
\begin{equation}
\sum_{j=1,2;\text{ }\mid k\mid>n}\mid a_{k,j}(t)\mid^{2}\leq m_{7}\left(
\sum_{j=1,2;\text{ }\mid k\mid>n}\left\vert \left(  f_{t},\left(
\begin{array}
[c]{c}%
1\\
(-1)^{j}i
\end{array}
\right)  e^{i((-1)^{j-1}2n+\overline{t})x}\right)  \right\vert ^{2}+\frac
{1}{n}\right)  , \tag{53}%
\end{equation}
where $a_{k,j}(t)$\ is defined in (14), and then the inequality
\begin{equation}
\parallel R_{n}(\cdot,t)\parallel^{2}\leq m_{8}\sum_{j=1,2;\text{ }\mid
k\mid>n}\mid a_{k,j}(t)\mid^{2}. \tag{54}%
\end{equation}
It follows from (50) that
\begin{equation}
\mid a_{k,j}(t)\mid^{2}\leq2\left(  \left\vert \left(  f_{t},\frac{1}%
{\sqrt{2\pi}\left\Vert e^{itx}\right\Vert }\left(
\begin{array}
[c]{c}%
1\\
(-1)^{j}i
\end{array}
\right)  e^{i((-1)^{j-1}2n+\overline{t})x}\right)  \right\vert ^{2}+\mid
(f_{t},h_{n,j,t}^{\ast})\mid^{2}\right)  , \tag{55}%
\end{equation}
for $j=1,2$ and $\mid k\mid>n.$ Since $f$ is a continuous, compactly supported
function, using the Schwarz inequality and (51), we obtain
\begin{equation}
\mid(f_{t},h_{n,j,t}^{\ast})\mid^{2}<m_{9}n^{-2}. \tag{56}%
\end{equation}
Therefore, (53) follows from (55).

Now we prove (54). Since the system
\[
\left\{  \frac{1}{\sqrt{2\pi}}\left(
\begin{array}
[c]{c}%
1\\
(-1)^{j}i
\end{array}
\right)  e^{i((-1)^{j-1}2n)x}:j=1,2;n\in\mathbb{Z}\right\}
\]
is an orthonormal basis, we obtain, by the Bessel inequality, that
\begin{equation}
\sum_{j=1,2;\text{ }\mid k\mid>n}\left\vert \left(  f_{t},\frac{1}{\sqrt{2\pi
}\left\Vert e^{itx}\right\Vert }\left(
\begin{array}
[c]{c}%
1\\
(-1)^{j}i
\end{array}
\right)  e^{i((-1)^{j-1}2n+\overline{t})x}\right)  \right\vert ^{2}%
\leq\parallel\frac{e^{itx}}{\left\Vert e^{itx}\right\Vert }f_{t}\parallel
^{2}<m_{10} \tag{57}%
\end{equation}
for all $t\in E(h)\cup\gamma(0,h)\cup\gamma(1,h).$ Hence, it follows from (57)
and (53) that
\[
\sum_{j=1,2;\text{ }\mid k\mid>n}\mid a_{k,j}(t)\mid^{2}\leq m_{11}.
\]
On the other hand, it also follows from (51) that
\[
\left\Vert a_{k,j}(t)h_{k,j,t}\right\Vert \leq m_{12}(\mid a_{k,j}(t)\mid
^{2}+n^{-2}).
\]
Therefore, the series
\[
\Sigma_{1}:=\sum_{j=1,2;\text{ }\mid k\mid>n}a_{k,j}(t)\frac{1}{\sqrt{2\pi
}\left\Vert e^{itx}\right\Vert }\left(
\begin{array}
[c]{c}%
1\\
(-1)^{j}i
\end{array}
\right)  e^{i((-1)^{j-1}2n+t)x}%
\]
and
\[
\Sigma_{2}:\text{ }\sum_{j=1,2;\text{ }\mid k\mid>n}a_{k,j}(t)h_{k,j,t}(x)
\]
converge in the norm of $L_{2}^{2}(0,\pi)$ and we have
\begin{equation}
\parallel R_{n}(\cdot,t)\parallel^{2}=\left\Vert \Sigma_{1}+\Sigma
_{2}\right\Vert ^{2}\leq2\left\Vert \Sigma_{1}\right\Vert ^{2}:+2\left\Vert
\Sigma_{2}\right\Vert ^{2}. \tag{58}%
\end{equation}
Arguing as in the proof (57), we obtain
\begin{equation}
\left\Vert \Sigma_{1}\right\Vert ^{2}\leq m_{13}\sum_{j=1,2;\text{ }\mid
k\mid>n}\mid a_{k,j}(t)\mid^{2}. \tag{59}%
\end{equation}
Now let us estimate $\left\Vert \Sigma_{2}\right\Vert .$ It follows from (51)
that
\[
\left\Vert \Sigma_{2}\right\Vert \leq m_{14}\sum_{j=1,2;\text{ }\mid k\mid
>n}\mid a_{k,j}(t)\mid\frac{1}{\mid n\mid}.
\]
Therefore, using the Schwarz inequality for $l_{2}^{2},$ we obtain
\begin{equation}
\left\Vert \Sigma_{2}\right\Vert ^{2}=\left(  \sum_{j=1,2;\text{ }\mid
k\mid>n}\mid a_{k,j}(t)\mid^{2}\right)  O(n^{-1}). \tag{60}%
\end{equation}
Thus, (54) follows from (58)-(60). The proof of the lemma follows from (53)
and (54).
\end{proof}

Now, instead of Lemma 2 of [16], using Lemma 2, and repeating the proofs of
Theorems 3 and 4 of [16], we obtain the following:

\begin{theorem}
For every compactly supported and continuous function $f$ the equalities
\begin{equation}
\int\limits_{E(h)}f_{t}(x)dt=\sum\limits_{j=1,2;k\in\mathbb{Z}}\int
\limits_{E(h)}a_{k,j}(t)\Psi_{k,j,t}(x)dt, \tag{61}%
\end{equation}%
\begin{equation}
\int\limits_{\gamma(0,h)}f_{t}(x)dt=\sum\limits_{j=1,2;k\in\mathbb{Z}}%
\int\limits_{\gamma(0,h)}a_{k,j}(t)\Psi_{k,j,t}(x)dt, \tag{62}%
\end{equation}%
\begin{equation}
\int\limits_{\gamma(1,h)}f_{t}(x)dt=\sum\limits_{j=1,2;k\in\mathbb{Z}}%
\int\limits_{\gamma(1,h)}a_{k,j}(t)\Psi_{k,j,t}(x)dt, \tag{63}%
\end{equation}
hold, where $0<h<\frac{1}{10}.$ The series in (61)-(63) converge in the norm
of $L_{2}^{2}(a,b)$ for every $a,b\in\mathbb{R}.$
\end{theorem}

Thus, \textbf{Step 2} is completed. In \textbf{Step 3}, we replace
$\gamma(0,h)$ and $\gamma(\pi,h)$ by $[-h,h]$ and $[1-h,1+h],$ respectively,
in the right hand sides of (62) and (63). By Corollary 1, only the periodic
and antiperiodic eigenvalues ($\lambda_{n,j}(0)$ and $\lambda_{n,j}(1)$) my
became ESS and hence the terms $a_{k,j}(t)\Psi_{k,j,t}(x)$ my become
non-integrable on $[-h,h]$ and $[1-h,1+h].$ Fortunately, for sufficiently
large $k,$ the sum $a_{k,1}(t)\Psi_{k,j,1}+a_{k,2}(t)\Psi_{k,2,t}(x),$ which
is the total projection of $L_{t}(Q)$ corresponding to the pair of neighboring
eigenvalues $\lambda_{k,1}(t)$ and $\lambda_{k,2}(t)$ is integrable on
$[-h,h]$.

To prove this and similar statements, we consider the total projection
\[
P_{t}\left(  \gamma\right)  f:=\int_{\gamma}\left(  L_{t}(Q)-\lambda I\right)
^{-1}fd\lambda
\]
defined by integral over closed curves $\gamma$ lying in the resolvent of
$L_{t}(Q)$. Let us find $\left(  L_{t}(Q)-\lambda I\right)  ^{-1}f$ $\ $in a
standard way by solving the Dirac equation
\[
Jy^{^{\prime}}(x)+Q(x)y-\lambda y=f,
\]
where $J=\left(
\begin{array}
[c]{cc}%
0 & 1\\
-1 & 0
\end{array}
\right)  .$ Applying the method of variation of constants, we obtain%
\[
y(x)=Y(x,\lambda)\alpha-Y(x,\lambda)\int_{0}^{x}Y^{-1}(s,\lambda)Jf(s)ds),
\]
where $Y(x,\lambda)=\left(
\begin{array}
[c]{cc}%
c_{1}(x,\lambda) & s_{1}(x,\lambda)\\
c_{2}(x,\lambda) & s_{2}(x,\lambda)
\end{array}
\right)  ,$ $Y^{-1}(x,\lambda)\left(
\begin{array}
[c]{cc}%
s_{2}(x,\lambda) & -s_{1}(x,\lambda)\\
-c_{2}(x,\lambda) & c_{1}(x,\lambda)
\end{array}
\right)  $ and $\alpha\in\mathbb{C}^{2}.$ From the boundary condition (2), we
see that%

\[
\alpha=(Y(\pi,\lambda)-e^{i\pi t}I)^{-1}Y(\pi,\lambda)\int_{0}^{\pi}%
Y^{-1}(s,\lambda)Jf(s)ds).
\]
Therefore, we have
\begin{align}
\left(  L_{t}(Q)-\lambda I\right)  ^{-1}f(x)  &  =Y(x,\lambda)(Y(\pi
,\lambda)-e^{i\pi t}I)^{-1}Y(\pi,\lambda)\int_{0}^{\pi}Y^{-1}(s,\lambda
)Jf(s)ds-\tag{64}\\
&  Y(x,\lambda)\int_{0}^{x}Y^{-1}(s,\lambda)Jf(s)ds,\nonumber
\end{align}
where
\begin{equation}
(Y(\pi,\lambda)-e^{i\pi t}I)^{-1}=\frac{1}{\Delta(\lambda,t)}\left(
\begin{array}
[c]{cc}%
s_{2}(\pi,\lambda)-e^{i\pi t} & -s_{1}(\pi,\lambda)\\
-c_{2}(\pi,\lambda) & c_{1}(\pi,\lambda)-e^{i\pi t}%
\end{array}
\right)  \tag{65}%
\end{equation}
and $\Delta(\lambda,t)=e^{it}(F(\lambda)-2\cos\pi t)$ (see (5)). Now we are
ready to replace the integrals over $\gamma(0,h)$ on the right side of (62) by
the integral over $[-h,h].$

\begin{theorem}
If $f$ is a continuous and compactly supported function,, then
\begin{equation}
\int\limits_{\gamma(0,h)}f_{t}dt=\int\limits_{-h,}^{h}%
{\textstyle\sum_{\substack{j=1,2;\\\text{ }\left\vert n\right\vert \leq
N(h)}}}
a_{n,j}(t)\Psi_{n,j,t}dt+\sum_{\left\vert n\right\vert >N(h)}\int
\limits_{-h,}^{h}a_{n,1}(t)\Psi_{n,1,t}+a_{n,2}(t)\Psi_{n,2,t}dt, \tag{66}%
\end{equation}
where $N(h)$ is defined in Theorem1. Moreover,%
\begin{equation}
\int\limits_{\lbrack-h,h]}\left(  \sum_{j=1,2;\text{ }\left\vert n\right\vert
\leq N(h)}a_{n,j}(t)\Psi_{n,j,t}\right)  dt=\lim_{\delta\rightarrow0}%
\sum_{\substack{j=1,2;\\\text{ }\left\vert n\right\vert \leq N(h)}%
}\int\limits_{\delta<\left\vert t\right\vert \leq h}a_{n,j}(t)\Psi_{n,j,t}dt
\tag{67}%
\end{equation}
and%
\begin{equation}
\int\limits_{\lbrack-h,h]}%
{\textstyle\sum_{\substack{j=1,2\\\text{ }}}}
a_{n,j}(t)\Psi_{n,j,t}dt=\lim_{\delta\rightarrow0}%
{\textstyle\sum_{j=1,2}}
\int\limits_{\delta<\left\vert t\right\vert \leq h}a_{n,j}(t)\Psi_{n,j,t}dt,
\tag{68}%
\end{equation}
for $\left\vert n\right\vert >N(h)$. In addition, if $\lambda_{n,j}(0),$ with
$\left\vert n\right\vert >N(h),$ is not an ESS, then
\begin{equation}
\int\limits_{\lbrack-h,h]}a_{n,1}(t)\Psi_{n,1,t}+a_{n,1}(t)\Psi_{n,2,t}%
dt=\int\limits_{[-h,h]}a_{n,1}(t)\Psi_{n,1,t}dt+\int\limits_{[-h,h]}%
a_{n,2}(t)\Psi_{n,2,t}dt. \tag{69}%
\end{equation}
The series in (66) converges in the norm of $L_{2}^{2}(a,b)$ for every
$a,b\in\mathbb{R}.$
\end{theorem}

\begin{proof}
It follows from Theorem1 that for $\left\vert n\right\vert >N(h)$ the circle
\[
C(n)=\left\{  z\in\mathbb{C}:\left\vert z-2n\right\vert =1\right\}
\]
contains inside only two eigenvalues (counting multiplicities) denoted by
$\lambda_{n,1}(t)$ and $\lambda_{n,t}(t)$ of the operators $L_{t}(Q)$ for
$|t|\leq h$. Moreover, $C(n)$ lies in the resolvent set of $L_{t}$ for
$|t|\leq h.$ Similarly, there exists a closed curve $\Gamma(0)$ such that
$\Gamma(0)$ lies in the resolvent set of $L_{t}$ for $|t|\leq h$ and all
eigenvalues of $L_{t}$ for $|t|\leq h$ that do not lie in $C(n)$ for $n>N(h)$
are contained in the region enclosed by $\Gamma(0).$ Consider the total
projections
\[
T_{C(n)}(x,t):=\int\limits_{C(n)}\left(  L_{t}(Q)-\lambda I\right)
^{-1}f(x)d\lambda\text{ }\And\text{ }T_{\Gamma(0)}(x,t):=\int\limits_{\Gamma
(0)}\left(  L_{t}(Q)-\lambda I\right)  ^{-1}f(x)d\lambda.
\]

Since $\Delta(\lambda,t)$ is continuous on the compact $C(n)\times
\overline{U_{h}(0)},$ where $\overline{U_{h}(0)}=\left\{  t\in\mathbb{C}%
:\left\vert t\right\vert \leq h\right\}  $, there exists a positive constant
$m_{15}$ such that $\left\vert \Delta(\lambda,t)\right\vert \geq m_{15}$ for
all $(\lambda,t)\in C(n)\times\overline{U_{h}(0)}$ and $(\lambda,t)\in
\Gamma(0)\times\overline{U_{h}(0)}.$ Using this inequality together with (64)
and (65), and taking into account that $f_{t}(x)$ is a finite sum, we conclude
that there exists a positive constant $m_{16}$ such that
\begin{equation}
\left\vert T_{C(n)}(x,t)\right\vert \leq m_{16},\text{ }\left\vert
T_{\Gamma(0)}(x,t)\right\vert \leq m_{16} \tag{70}%
\end{equation}
for all $(x,t)\in\lbrack0,\pi]\times\overline{U_{h}(0)}.$ Therefore, by
repeating the proof of Theorem 5 in [16], we complete the proof of this theorem.
\end{proof}

The next assertions are proved in the same way.

\begin{theorem}
If $f$ is a continuous and compactly supported function, then
\begin{equation}
\int\limits_{\gamma(1,h)}f_{t}(x)dt=\int\limits_{[\pi-h,\pi+h]}\sum
_{n=-N(h)-1}^{N(h)}\sum_{j=1,2}a_{n,j}(t)\Psi_{n,j,t}(x)dt+ \tag{71}%
\end{equation}%
\[
\sum_{n\notin(-N(h)-1,N(h))}\int\limits_{[\pi-h,\pi+h]}\left(  a_{n,1}%
(t)\Psi_{n,t}(x)+a_{n+1,2}(t)\Psi_{n+1,2,t}(x)\right)  dt.
\]
Moreover,%
\[
\int\limits_{\lbrack\pi-h,\pi+h]}\sum_{n=-N(h)-1}^{N(h)}\sum_{j=1,2}%
a_{n,j}(t)\Psi_{n,t}dt=\lim_{\delta\rightarrow0}\sum_{n=-N(h)-1}^{N(h)}%
\sum_{j=1,2}\int\limits_{\delta<\left\vert \pi-t\right\vert \leq h}%
a_{n,j}(t)\Psi_{n,j,t}dt,
\]
and%
\begin{align}
&  \int\limits_{[\pi-h,\pi+h]}a_{n,1}(t)\Psi_{n,1,t}+a_{(n+1),2}%
(t)\Psi_{(n+1),2,t}dt\tag{72}\\
&  =\lim_{\delta\rightarrow0}\left(  \int\limits_{\delta<\left\vert
\pi-t\right\vert \leq h}a_{n,1}(t)\Psi_{n,1,t}dt+\int\limits_{\delta
<\left\vert \pi-t\right\vert \leq h}a_{n+1,2}(t)\Psi_{n+1,2,t}dt\right)
,\nonumber
\end{align}
for $n\notin(-N(h)-1,N(h)).$ In addition, if $\lambda_{n,j}(\pi)$ is not an
ESS, then%
\[
\int\limits_{\lbrack\pi-h,\pi+h]}a_{n,1}(t)\Psi_{n,1,t}+a_{(n+1),2}%
(t)\Psi_{(n+1),2,t}dt=\int\limits_{[\pi-h,\pi+h]}a_{n,1}(t)\Psi_{n,1,t}%
dt+\int\limits_{[\pi-h,\pi+h]}a_{n+1,2}(t)\Psi_{n+1,2,t}dt.
\]
The series in (71) converge in the norm of $L_{2}^{2}(a,b)$ for every
$a,b\in\mathbb{R}.$
\end{theorem}

Now, to obtain the final spectral expansion theorem in terms of the
quasimomentum $t,$ from Theorems 5-7, we introduce some notation and
definitions. Let $\Lambda_{s}(0)$ for $s=1,2,...,s_{0}$ and $\Lambda_{s}(1)$
for $s=1,2,...,s_{1}$ be, respectively, the periodic and antiperiodic
eigenvalues, which are the ESS of $L(Q),$ lying in the sets
\[
\mathbb{E}(0):=\left\{  \lambda_{n,j}(0):j=1,2;\text{ }\left\vert n\right\vert
\leq N(h)\right\}  \text{ }%
\]
and
\[
\mathbb{E}(1):=\left\{  \lambda_{n,j}(1):j=1,2;\text{ }-N(h)-1\leq n\leq
N(h)\right\}  .
\]
Introduce the following notation:
\[
\mathbb{T}(\Lambda_{s}(0))=:\left\{  \left(  n,j\right)  :\left\vert
n\right\vert \leq N(h),\text{ }j=1,2,\text{ }\lambda_{n,j}(0)=\Lambda
_{s}(0)\right\}  ,\text{ }%
\]%
\[
\mathbb{T}(\Lambda_{s}(1))=:\left\{  \left(  n,j\right)  :-N(h)-1\leq n\leq
N(h),\text{ }j=1,2,\text{ }\lambda_{n,j}(1)=\Lambda_{s}(1)\right\}
\]
and
\[
\mathbb{T}(0)\mathbb{=}%
{\textstyle\bigcup\limits_{s=1,2,...,s_{0}}}
\mathbb{T}(\Lambda_{s}(0)),\text{ }\mathbb{T}(1)\mathbb{=}%
{\textstyle\bigcup\limits_{s=1,2,...,s_{1}}}
\mathbb{T}(\Lambda_{s}(1)).
\]
With this notations, the eigenvalues $\lambda_{n,j}(0)$ and $\lambda_{k,j}(1)$
from $\mathbb{E}(0)$ and $\mathbb{E}(1),$ respectively, are ESS if and only if
the pairs $\left(  n,j\right)  $ and $\left(  k,j\right)  $ belong to
$\mathbb{T}(0)$ and $\mathbb{T}(1).$ In other words, $\lambda_{n,j}(0)$ and
$\lambda_{k,j}(1)$ are not ESS if the corresponding index pairs belong to
\[
\mathbb{K}(0):=\left\{  \left(  n,j\right)  :j=1,2;\text{ }\left\vert
n\right\vert \leq N(h)\right\}  \backslash\mathbb{T}(0)\text{ }%
\]
and
\[
\mathbb{K}(1):=\left\{  \left(  n,j\right)  :j=1,2;\text{ }-N(h)-1\leq n\leq
N(h)\right\}  \backslash\mathbb{T}(1),
\]
respectively. Using this notation, we prove the following spectral expansion.

\begin{theorem}
For every compactly supported and continuous function $f$ the following
spectral expansion holds:
\begin{equation}
f=%
{\displaystyle\sum\limits_{j=1,2;\text{ }n\in\mathbb{Z}}}
\int\limits_{B(h)}a_{n,j}(t)\Psi_{n,j,t}dt+\sum_{\left(  n,j\right)
\in\mathbb{K}(0)}\int\limits_{[-h,h]}a_{n,j}(t)\Psi_{n,j,t}dt+\sum_{\left(
n,j\right)  \in\mathbb{K}(1)}\int\limits_{[1-h,1+h]}a_{n,j}(t)\Psi_{n,j,t}dt+
\tag{73}%
\end{equation}%
\[
\sum_{\left\vert n\right\vert >N}\int\limits_{[-h,h]}\left[  a_{n,1}%
(t)\Psi_{n,1,t}+a_{n,2}(t)\Psi_{n,2,t}\right]  dt+\sum_{n\notin(-N(h)-1,N(h))}%
\int\limits_{[1-h,1+h]}\left[  a_{n,1}(t)\Psi_{n,1,t}+a_{n+1,2}(t)\Psi
_{n+1,2,t}\right]  dt+
\]%
\[
\sum_{s=1}^{s_{0}}\int\limits_{[-h,h]}\sum_{\left(  n,j\right)  \in
\mathbb{T}(\Lambda_{s}(0))}a_{n,j}(t)\Psi_{n,j,t}dt+\sum_{s=1}^{s_{1}}%
\int\limits_{[1-h,1+h]}\sum_{n\in\mathbb{T}(\Lambda_{s}(1))}a_{n,j}%
(t)\Psi_{n,j,t}.
\]
The series in (73) converge in the norm of $L_{2}^{2}(a,b)$ for every
$a,b\in\mathbb{R}.$

Moreover, (68), (72) and the following equalities hold
\begin{equation}
\int\limits_{\lbrack-h,h]}\sum_{\left(  n,j\right)  \in\mathbb{T}(\Lambda
_{s}(0))}a_{n,j}(t)\Psi_{n,j,t}dt=\lim_{\delta\rightarrow0}\sum_{\left(
n,j\right)  \in\mathbb{T}(\Lambda_{s}(0))}\int\limits_{\delta<\left\vert
t\right\vert \leq h}a_{n,j}(t)\Psi_{n,j,t}(x)dt, \tag{74}%
\end{equation}%
\begin{equation}
\int\limits_{\lbrack1-h,1+h]}\sum_{\left(  n,j\right)  \in\mathbb{T}%
(\Lambda_{s}(1))}a_{n,j}(t)\Psi_{n,j,t}dt=\lim_{\delta\rightarrow0}%
\sum_{\left(  n,j\right)  \in\mathbb{T}(\Lambda_{s}(1))}\int\limits_{\delta
<\left\vert t-1\right\vert \leq h}a_{n,1}(t)\Psi_{n,1,t}. \tag{75}%
\end{equation}

\end{theorem}

\begin{proof}
By Theorem 5, 6 and 7 the first fourth and fifth integrals in the right side
of (73) exist. By Theorem 4, it follows from the definitions of $\mathbb{K}%
(0)$ and $\mathbb{K}(1)$ that second and third integrals also exist. We now
prove that the sixth integral exists. Namely, we show that if $f$ is a
continuous and compactly supported function, then the expression%
\begin{equation}
\sum_{\left(  n,j\right)  \in\mathbb{T}(\Lambda_{s}(0))}a_{n,j}(t)\Psi_{n,j,t}
\tag{76}%
\end{equation}
is integrable over $[-h,h].$ If $\Lambda_{s}(0)$ is a $p$-multiple eigenvalue
of $L_{0}(Q)$ for some $p\geq2,$ then it is a $p$-multiple root of equation
(6) for $t=0$. Therefore, by the implicit function theorem, there exist a disk
$U_{\varepsilon}(0)=\{t\in\mathbb{C}:\left\vert t\right\vert <\varepsilon\}$
and constants $r_{2}>r_{1}>0$ such that the followings hold:

$(i)$ In the $r_{1}$-neighborhood of $\Lambda_{s}(0)$ the operators $L_{t}$
for $t\in U_{\varepsilon}(0)\backslash\{0\}$ have exactly $p$ simple
eigenvalues, namely $\lambda_{n,j}(t)$ for $\left(  n,j\right)  \in
\mathbb{T}(\Lambda_{s}(0)).$

$(ii)$ In the $r_{2}$-neighborhood of $\Lambda_{s}(0)$ the operators $L_{t}$
for $t=0$ and $t\in U_{\varepsilon}(0)\backslash\{0\}$ have no eigenvalues
other than $\Lambda_{s}(0)$ and $\lambda_{n,j}(t)$ for $\left(  n,j\right)
\in\mathbb{T}(\Lambda_{s}(0)),$ respectively.

Let $C$ be the circle centered at $\Lambda_{j}(0)$ with radius $r,$ where
$r_{1}<r<$ $r_{2}.$ Replacing $C(n)$ by $C$, and repeating the proof of (70),
we obtain that there exists a positive constant $m_{17}$ such that $\left\vert
T_{C}(x,t)\right\vert \leq m_{17}$ for all $(x,t)\in\lbrack0,\pi
]\times\overline{U_{h}(0)}.$ Therefore, arguing as in the proof of Theorem 6,
we obtain that (76) is integrable on $(-\varepsilon,\varepsilon).$ Together
with Corollary 2 this implies that (76) is integrable on $[-h,h]$ and that
(74) holds. Consequently, the sixth integral (73) exists. In the same way, we
proof that seventh integral exists and that (75) holds. The equalities (68)
and (72) were proved in Theorem 6 and 7, respectively. Now, (73) follows from
(43), (61), (66) and (71).
\end{proof}

\begin{conclusion}
Let us explain why the spectral expansions formula has the form (73). The
expression $a_{n,j}(t)\Psi_{n,j,t}(x)$ is integrable on $B(h),$ because
$\lambda_{n,j}(t)$ is not ESS for $t\in B(h).$ Therefore, in the first
summation in (73), there is no need to combine any of these terms. Moreover,
these expressions are integrable on $[-h,h]$ and $[1-h,1+h],$ if $\left(
n,j\right)  \in\mathbb{K}(0)$ and $\left(  n,j\right)  \in\mathbb{K}(1),$
respectively (see the definition of $\mathbb{K}(0)$ and $\mathbb{K}(1)$), by
the same reason. Hence, in the second and third summations in (73), we also do
not need to combine any terms. However, if $\left(  n,j\right)  \in
\mathbb{T}(\Lambda_{s}(0)),$ then $\lambda_{n,j}(0)=\Lambda_{s}(0)$ is an ESS
and the term $a_{n,j}(t)\Psi_{n,j,t}(x)$ is not integrable over $[-h,h].$
Fortunately the sum of these terms over $\left(  n,j\right)  \in
\mathbb{T}(\Lambda_{j}(0))$ is integrable. That is why we need combine these
terms (putting the terms together) in the sixth summation of (73). Note that,
this sum is integrable, because it represents the total projection of
$L_{t}(Q)$ corresponding to the eigenvalues $\left\{  \lambda_{n,j}(t):\left(
n,j\right)  \in\mathbb{T}(\Lambda_{j}(0))\right\}  $ and the continuity of the
total projection is a natural occurrence in the perturbation theory. An
analogous situation occurs in the seventh summation in (73). If $\left\vert
n\right\vert >N,$ then the multiplicity of the eigenvalue $\lambda_{n,j}(0)$
does not exceed $2,$ and it is double eigenvalue if and only if $\lambda
_{n,1}(0)=\lambda_{n,2}(0).$ In this case the double eigenvalue may become an
ESS. Then both terms $a_{n,1}(t)\Psi_{n,1,t}$ and $a_{n,2}(t)\Psi_{n,2,t}$ are
not integrable over $[-h,h],$ whereas their sum is integrable. That is why we
combine these two terms in the fourth summations of (73). The same reasoning
applies to the fifth summation of (73).
\end{conclusion}


\begin{thebibliography}{99}                                                                                               %


\bibitem {}P. Djakov and B. Mityagin, Bari-Markus property for Riesz
projections of 1D periodic Dirac operators,\ Math. Nachr., 283, No. 3,
443--462 (2010).

\bibitem {}P. Djakov and B. Mityagin, Criteria for existence of Riesz bases
consisting of root functions of Hill and 1D Dirac operators, J. Funct. Anal.,
263, No. 8, 2300--2332 (2012).

\bibitem {}P. Djakov and B. Mityagin, Unconditional convergence of spectral
decompositions of 1DDirac operators with regular boundary conditions,\ Indiana
Univ. Math. J., 61, No. 1, 359--398 (2012).

\bibitem {}I. M. Gelfand, Expansion in series of eigenfunctions of an equation
with periodic coefficients, Sov. Math. Dokl. {73}, 1117-1120 (1950).

\bibitem {}B. M. Levitan, I. S. Sargsyan, Spectrum and Trace of Ordinary
Differential Operators, MSU press, Moscow 2003.

\bibitem {}A. A. Lunyov and M. M. Malamud, On the Riesz basis property of the
root vector systemfor Dirac-type $2\times\ 2$ systems,\ Dokl. Math., 90, No.
2, 556--561 (2014).

\bibitem {}A. A. Lunyov and M. M. Malamud, On the completeness and Riesz basis
property ofroot subspaces of boundary value problems for first order systems
and applications,\ J. Spectral Theory, 5, No. 1, 17--70 (2015).

\bibitem {}A. A. Lunyov and M. M. Malamud, On the Riesz basis property of root
vectors systemfor $2\times\ 2$ Dirac type operators,\ J. Math. Anal. Appl.,
441, 57--103 (2016).

\bibitem {}A. A. Lunyov, M. M. Malamud, On the formula for characteristic
determinants of boundary value problems for $n\times\ n$ Dirac type systems
and its applications, Advances in Mathematics 478, 110389 (2025).

\bibitem {}D. C. McGarvey, Differential operators with periodic coefficients
in $L_{p}(-\infty,\infty)$, Journal of Mathematical Analysis and Applications,
{11}, 564-596 (1965).

\bibitem {}D. McGarvey, Linear differential systems with periodic coefficients
involving a large parameter, \textit{J. }Diff. Eq.\textit{,} \textbf{2},
115--142 (1966).

\bibitem {}A. M. Savchuk and I. V. Sadovnichaya, Asymptotic Formulas for
Fundamental Solutions of the Dirac Sistem with Complex-Valued Integrable
Potential, Differential Equations, 40 No. 5, 545-556 (2013).

\bibitem {}A. M. Savchuk and A. A. Shkalikov, The Dirac operator with
complex-valued summablepotential,\ Math. Notes, 96, No. 5--6, 777--810 (2014).

\bibitem {}A. A. Shkalikov, Regular spectral problems for systems of ordinary
differential equations of the first order, Russian Math. Surveys, 76:5,
939--941 (2021).

\bibitem {}V. A.\ Tkachenko, Non-self-adjoint Periodic Dirac Operator,
Operator Theory: Advances and Application Vol. 121, 485-512 (2001).

\bibitem {}O. A. Veliev, Essential spectral singularities and the spectral
expantion for the Hill operator, Communucation on Pure and Applied analysis,
16, No. 6, 2227-2251 (2017).

\bibitem {}O. A. Veliev, Asymptotically spectral periodic differential
operators, Mathematical Notes, Vol. 104, No. 3, 364--376 (2018).

\bibitem {}O. A. Veliev, Spectral expansion series with parenthesis for the
nonself-adjoint periodic differential operators, Communucation on Pure and
Applied analysis, 18, No. 1, 397-424 (2019).

\bibitem {}O. A. Veliev, Spectral Expansion for the Non-self-adjoint
Differential Operators with the Periodic Matrix Coefficients, Mathematical
Notes, Vol.112, No. 6, 1025-1043, (2022).

\bibitem {}O. A. Veliev, Non-Self-Adjoint Schr\"{o}dinger Operator with a
Periodic Potential: Spectral Theories for Scalar and Vectorial Cases and Their
Generalizations, Springer, Switzerland, 487p., 2025.

\bibitem {}Zygmund, A., Trigonometric Series, 2 volumes (Cambridge University Press)
\end{thebibliography}
\end{document}